\documentclass[12pt]{amsart}
\usepackage{amsthm, amssymb, amsmath}
\usepackage{mdwlist}
\usepackage{mathrsfs}
\usepackage{a4wide}

\newtheorem{theorem}{Theorem}
\newtheorem{proposition}{Proposition}

\theoremstyle{definition}
\newtheorem*{definition}{Definition}
\newtheorem{example}{Example}

\theoremstyle{remark}

\newcommand{\N}{\mathbb{N}}
\newcommand{\F}{\mathcal{F}}
\newcommand{\G}{\mathcal{G}}
\newcommand{\Pa}{\mathcal{P}}
\newcommand{\flim}{\mathcal{F}\!\!-\!\!\lim}
\newcommand{\e}{\varepsilon}

\newcommand{\w}{\widetilde}
\newcommand{\abs}[1]{\lvert#1\rvert}
\newcommand{\n}[1]{\|#1\|}
\newcommand{\supp}{\mathrm{supp}_{(f_n)}}

\begin{document}
\title[$\F$-bases with brackets]{$\F$-bases with brackets and with individual\\ brackets in Banach spaces}

\author[T. Kochanek]{Tomasz Kochanek}
\address{Institute of Mathematics, University of Silesia, Bankowa 14, 40-007 Katowice, Poland}
\curraddr{}
\email{t\_kochanek@wp.pl}
\thanks{}

\date{}
\dedicatory{}

\begin{abstract}
We provide a partial answer to the question of Vladimir Kadets whether given an $\F$-basis of a~Banach space $X$, with respect to some filter $\F\subset\Pa(\N)$, the coordinate functionals are continuous. The answer is positive if the character of $\F$ is less than $\mathfrak{p}$. In this case every $\F$-basis is an~$M$-basis with brackets which are determined by an element of $\F$.
\end{abstract}

\maketitle

\section{Introduction}
Given any filter $\F$ of subsets of $\N$ and a~Banach space $X$ we say that a sequence $(e_n)_{n=1}^\infty$ is an $\F$-{\it basis} for $X$ if and only if for each $x\in X$ there is a unique sequence of scalars $(a_n)_{n=1}^\infty$ such that $$x=\flim_{n\to\infty}\sum_{k=1}^na_ke_k$$in the norm topology of $X$ (i.e. for each $\e>0$ there is a~set $A\in\F$ such that $\n{x-\sum_{k=1}^na_ke_k}<\e$ for $n\in A$). In such a case we shall define the coordinate functionals by $e_n^\ast(x)=a_n$ and the partial sum projections by $S_n(x)=\sum_{k=1}^ne_k^\ast(x)e_k$ for $n\in\N$. Of course, all these maps are linear. 

The present paper is motivated by a~question posed by V.~Kadets during the $4$th conference {\it Integration, Vector Measures and Related Topics} who asked whether it is true in general that $e_n^\ast$ are continuous. Let us note that continuity of coordinate functionals has been usually included in the definition of $\F$-basis (cf. \cite{connor_ganichev_kadets}, \cite{ganichev_kadets}). Of particular interest is the case where $\F=\F_{st}$ is the filter of statistical convergence defined by $$\F_{st}=\Bigl\{A\subset\N\colon\lim_{n\to\infty}\frac{1}{n}\abs{A\cap\{1,\ldots ,n\}}=1\Bigr\}.$$

For any filter $\F$ of subsets of $\N$ let $\chi(\F)$ stand for its {\it character}, that is, the minimal cardinality of a~subfamily of $\F$ which generates $\F$: $$\chi(\F)=\min\bigl\{\abs{\mathcal{B}}\colon\mathcal{B}\subset\F,\,\,\forall_{A\in\F}\,\exists_{B\in\mathcal{B}}\,\,B\subseteq A\bigr\} .$$We will show that the answer to Kadets' question is positive when the character of $\F$ is less than $\mathfrak{p}$, the {\it pseudointersection number}, which is the least cardinal number such that $P(\mathfrak{p}^+)$ is false, where $P(\kappa)$ is the following statement:
\begin{itemize*}
\item[$P(\kappa)$:] If $\mathscr{A}$ is a~family of subsets of $\N$ such that $\abs{\mathscr{A}}<\kappa$ and $A_1\cap\ldots\cap A_k$ is infinite for any $A_1,\ldots ,A_k\in\mathscr{A}$, then there is an~infinite set $B\subset\N$ such that $B\setminus A$ is finite for each $A\in\mathscr{A}$. 
\end{itemize*}

It is known that $\omega_1\leq\mathfrak{p}\leq\mathfrak{c}$ and that $\mathfrak{p}=\mathfrak{c}$ provided we assume Martin's axiom (this is known as Booth's lemma; cf. \cite[Theorem 11C]{fremlin}). We thus obtain continuity of coordinate functionals associated with $\F$-bases for which $\chi(\F)\leq\omega$ (i.e. $\F$ is countably generated) and, under Martin's axiom, for which $\chi(\F)<\mathfrak{c}$. In fact, we will see (Theorem \ref{T1}) that any $\F$-basis $(e_n)_{n=1}^\infty$ of a~Banach space $X$, with $\chi(\F)<\mathfrak{p}$, is an~$M$-{\it basis with brackets} (cf. \cite{kadets}), that is, there is a~sequence $n_1<n_2<\ldots$ of natural numbers such that for each $x\in X$ we have 
\begin{equation*}
x=\lim_{k\to\infty}\sum_{j=1}^{n_k}e_j^\ast(x)e_j.
\end{equation*}
Of course, every such basis generates a~finite-dimensional Schauder decomposition of $X$. The inequality $\chi(\F)<\mathfrak{p}$ also implies that the set $\{n_1<n_2<\ldots\}$ may be required to be a~member of $\F$.

My first proof of continuity of coordinate functionals was working for countably generated filters and it was D.H.~Fremlin who indicated that the argument should go through for some models where the Baire Category Theorem is valid for uncountably many meagre sets. This led me to the condition $\chi(\F)<\mathfrak{p}$.

In Section 3 we introduce the notion of $\F$-{\it basis with individual brackets}, analogues to the one of $M$-{\it basis with individual brackets} (see definitions therein), and we show that many $\F$-bases which arise naturally from Schauder bases belong to this class (Theorem \ref{PP}).

It should be mentioned that since the statistical filter $\F_{st}$ is {\it tall} (i.e. every infinite subset of $\N$ contains an~infinite subset belonging to the dual ideal of $\F_{st}$), we have $\chi(\F_{st})\geq\mathfrak{p}$, so the question posed by Kadets remains unanswered in the case $\F=\F_{st}$.

\section{Continuity of coordinate functionals}
Hereinafter $\F$ stands for a~filter of subsets of $\N$ and $(e_n)_{n=1}^\infty$ is an~$\F$-basis of a~Banach space $X$ unless otherwise stated. The coordinate functionals and the partial sum projections corresponding to $(e_n)_{n=1}^\infty$ will be denoted by $e_n^\ast$ and $S_n$.

\begin{proposition}\label{L1}
For every $A\in\F$ the space $$X_A=\bigl\{x\in X:\,\sup_{\nu\in A}\n{S_\nu(x)}<\infty\bigr\},$$equipped with a~norm $\n{\cdot}_A$ defined by $$\n{x}_A=\sup_{\nu\in A}\n{S_\nu(x)},$$is a~Banach space. 
\end{proposition}
\begin{proof}
First, observe that $\n{\cdot}_A$ is indeed a~norm on $X_A$. Homogeneity and the triangle inequality are trivial. Moreover, if $\n{x}_A=0$ then for each $\e>0$ one may find $B\in\F$ such that $\n{S_\nu(x)-x}<\e$ for $\nu\in B$, but then for each $\nu\in A\cap B$, which is non-empty as an~element of $\F$, we have $S_\nu(x)=0$. Hence $\n{x}<\e$ and consequently $x=0$.

Now, assume $(x_n)_{n=1}^\infty$ is a Cauchy sequence in $(X_A,\n{\cdot}_A)$. Then for every $\e>0$ one may find $m\in\N$ such that $$\n{S_\nu(x_m-x_n)}<\e/3\quad\mbox{for each }n\geq m\mbox{ and }\nu\in A.$$ We may choose $\nu$ in such a way that $\n{S_\nu(x_m)-x_m}<\e/3$ and $\n{S_\nu(x_n)-x_n}<\e/3$. These three inequalities give $\n{x_m-x_n}<\e$, which shows that $(x_n)_{n=1}^\infty$ is a Cauchy sequence in $(X_A,\n{\cdot})$. Therefore, there exists $x_0$ in the $\n{\cdot}$-closure of $X_A$ such that
\begin{equation}\label{x_0}
\lim_{n\to\infty}\n{x_n-x_0}=0.
\end{equation}
Similarly, for every $\nu\in A$ and $m,n\in\N$ we have $$\n{S_\nu(x_m)-S_\nu(x_n)}=\n{S_\nu(x_m-x_n)}\leq\n{x_m-x_n}_A,$$which shows that $(S_\nu(x_n))_{n=1}^\infty$ is a Cauchy sequence in $(X,\n{\cdot})$, and each of its elements lies in $\mathrm{span}\{e_j\}_{j\leq\nu}$. Hence, there is $y_\nu\in\mathrm{span}\{e_j\}_{j\leq\nu}$ such that
\begin{equation}\label{y_nu}
\lim_{n\to\infty}\n{S_\nu(x_n)-y_\nu}=0.
\end{equation}

For every $j\in\N$ denote $\alpha_j=e_j^\ast(y_\nu)$ for any $\nu\in A$, $j\leq\nu$. This definition does not depend on the choice of such a $\nu$. Indeed, if $k,\ell\in A$ satisfy $j\leq k\leq\ell$, then the continuity of $e_j^\ast$ on the finite-dimensional subspace $\mathrm{span}\{e_i\}_{i\leq\ell}$ gives $$e_j^\ast(y_k)=e_j^\ast\bigl(\lim_{n\to\infty}S_k(x_n)\bigr)=\lim_{n\to\infty}e_j^\ast(S_k(x_n))=\lim_{n\to\infty}e_j^\ast(S_\ell(x_n))=e_j^\ast(y_\ell).$$
We shall show that $$x_0=\F\!-\!\sum_{n=1}^\infty\alpha_ne_n,$$thus, in particular, $S_\nu(x_0)=y_\nu$ for every $\nu\in A$. To this end fix any $\e>0$ and choose $m\in\N$ such that for each $n\geq m$ we have $\n{S_\nu(x_m)-S_\nu(x_n)}<\e/3$ (for any $\nu\in A$) and $\n{x_m-x_n}<\e/3$. Now, let $B\in\F$ be such that for each $\nu\in B$ we have $\n{S_\nu(x_m)-x_m}<\e/3$. Then $A\cap B\in\F$ and for every $\nu\in A\cap B$ we get:
\begin{equation*}
\begin{split}
\n{y_\nu-x_0} &=\bigl\|\lim_{n\to\infty}S_\nu(x_n)-\lim_{n\to\infty}x_n\bigr\|\\
&\leq\lim_{n\to\infty}\n{S_\nu(x_m)-S_\nu(x_n)}+\n{S_\nu(x_m)-x_m}+\lim_{n\to\infty}\n{x_m-x_n}\leq\e ,
\end{split}
\end{equation*}
in view of \eqref{x_0} and \eqref{y_nu}. This shows that $$x_0=\flim_{\substack{\nu\to\infty\\ \nu\in A}}y_\nu.$$ Moreover, a similar estimate, for an arbitrary $\nu\in A$ and $m\in\N$ chosen as above, yields $$\n{y_\nu}\leq\n{x_0}+{1\over 3}\e+\n{S_\nu(x_m)}+\n{x_m}+{1\over 3}\e\leq {2\over 3}\e+\n{x_0}+\n{x_m}_A+\n{x_m},$$which implies $$\sup_{\nu\in A}\n{S_\nu(x_0)}=\sup_{\nu\in A}\n{y_\nu}<\infty ,$$thus $x_0\in X_A$. Now, for any $n\in\N$ we have 
\begin{equation*}
\begin{split}
\n{x_n-x_0}_A &=\sup_{\nu\in A}\n{S_\nu(x_n)-S_\nu(x_0)}=\sup_{\nu\in A}\n{S_\nu(x_n)-\lim_{m\to\infty}S_\nu(x_m)}\\
&\leq\limsup_{m\to\infty}\sup_{\nu\in A}\n{S_\nu(x_n)-S_\nu(x_m)},
\end{split}
\end{equation*}
which shows that $\lim_{n\to\infty}\n{x_n-x_0}_A=0$ and, consequently, $(X_A,\n{\cdot}_A)$ is a~Banach space.
\end{proof}

\begin{proposition}\label{L2}
If $\chi(\F)<\mathfrak{p}$ then there exists a~set $A\in\F$ such that $X_A=X$.
\end{proposition}
\begin{proof}
For any $A\in\F$ the identity mapping $i_A\colon (X_A,\n{\cdot}_A)\to (X,\n{\cdot})$ is continuous, since $\n{\cdot}_A\geq\n{\cdot}$. By Proposition \ref{L1} and the Open Mapping Theorem, either $i_A$ is surjective, or its image $X_A$ is a~meagre subset of $(X,\n{\cdot})$.

Let $\{A_\alpha\}_{\alpha<\chi(\F)}\subset\F$ be a~family generating $\F$. For every $x\in X$ there exists a~set $B\in\F$ such that $\sup_{n\in B}\n{S_n(x)}<\infty$, so $x\in X_{A_\alpha}$ for some $\alpha<\chi(F)$. Therefore, $$X=\bigcup_{\alpha<\chi(\F)}X_{A_\alpha},$$ and since the Baire Category Theorem is valid for less than $\mathfrak{p}$ meagre sets in any Polish space (cf. \cite[\S 22C]{fremlin}), not all the subspaces $X_{A_\alpha}$ may be meagre in $(X,\n{\cdot})$. Consequently, there is a~set $A\in\F$ with $X_A=X$.
\end{proof}

\begin{example}\label{exx}
A slight modification of \cite[Example 1]{connor_ganichev_kadets} shows that in general one may not expect that $X_A=X$ for some $A\in\F$. Namely, let $(e_n)_{n=1}^\infty$ be the canonical basis of $\ell_2$ with the coordinate functionals $(e_n^\ast)_{n=1}^\infty$. Put also $x_n=\sum_{i=1}^ne_i$. Then, as it is shown in \cite{connor_ganichev_kadets}, $(x_n)_{n=1}^\infty$ is an $\F_{st}$-basis of $\ell_2$ with the coordinate functionals given by $x_n^\ast=e_n^\ast-e_{n+1}^\ast$. They are, of course, continuous but for any increasing sequence $n_1<n_2<\ldots$ of natural numbers we may define an element $x=\sum_{k=1}^\infty a_ke_k$ of $\ell_2$ such that $\sup_{k\in\N}\n{S_{n_k}(x)}=\infty$.

To this end choose an increasing subsequence $(m_j)_{j=1}^\infty$ of $(n_j)_{j=1}^\infty$ with $m_j>j^4$ and put
$$a_k=\left\{\begin{array}{ll}1/\sqrt[4]{k} & \mbox{if there is }j\in\N\mbox{ such that }k=m_j+1,\\
0 & \mbox{otherwise}.\end{array}\right.$$Repeating the argument from \cite[Example 1]{connor_ganichev_kadets} we get our claim, which shows that in this case $(\ell_2)_A\subsetneq\ell_2$ for every infinite set $A\subset\N$ (not only for every $A\in\F_{st}$).
\end{example}

\begin{theorem}\label{T1}
If $\chi(\F)<\mathfrak{p}$ then any $\F$-basis is an~$M$-basis with brackets and all the coordinate functionals are continuous. Moreover, the equality
\begin{equation}\label{xnk}
x=\lim_{k\to\infty}\sum_{j=1}^{n_k}e_j^\ast(x)e_j
\end{equation}
holds true for each $x\in X$, where the sequence $n_1<n_2<\ldots$ may be chosen in such a~way that $\{n_1,n_2,\ldots\}\in\F$.
\end{theorem}
\begin{proof}
We may assume that $\F$ does not contain any finite sets, since otherwise $X$ would be finite-dimensional.

Let $A\in\F$ satisfy $X_A=X$. Applying the Open Mapping Theorem to the operator $i_A\colon (X_A,\n{\cdot}_A)\to (X,\n{\cdot})$ we infer that the inverse operator $i_A^{-1}$ is bounded, i.e. there is a~constant $K<\infty$ such that $\n{S_\nu(x)}\leq K\n{x}$ for all $x\in X$ and $\nu\in A$. This easily implies that all the coordinate functionals are continuous.

Indeed, fix any $j\in\N$ and suppose, in search of a~contradiction, that there is a sequence $(x_n)_{n=1}^\infty$ of elements of $X$ such that $\n{x_n}=1$ for $n\in\N$ and $e_j^\ast(x_n)\to\infty$. Pick any $\nu\in A$, $\nu\geq j$. Obviously, $e_1,\ldots ,e_\nu$ are linearly independent and since the finite-dimensional subspace $\mathrm{span}\{e_i\}_{i\leq\nu ,i\not=j}$ is closed, we infer that $$\delta :=\inf\bigl\{\n{e_j+y}: y\in\mathrm{span}\{e_i\}_{i\leq\nu ,i\not=j}\bigr\}>0.$$Since $$S_\nu(x_n)=e_j^\ast(x_n)e_j+\sum_{i=1, i\not=j}^\nu e_i^\ast(x_n)e_i,$$ we have $$\n{S_\nu(x_n)}\geq\delta\cdot\abs{e_j^\ast(x_n)}\xrightarrow[n\to\infty]{}\infty ,$$which contradicts the fact that $S_\nu$ is continuous.

Now, in order to show that $(e_n)_{n=1}^\infty$ is an $M$-basis with brackets one may simply use the definition of $\mathfrak{p}$ to produce an infinite set $B\subset\N$ such that $\abs{B\setminus A_\alpha}<\omega$ for every $A_\alpha$ from some fixed (centered) family $(A_\alpha)_{\alpha<\chi(\F)}$ generating $\F$. Then $S_\nu(x)\to x$ for every $x\in X$ as $\nu\in B$, $\nu\to\infty$. However, there is no reason why $B$ should be an element of $\F$. Instead one may use the following argument for which I am grateful to Vladimir Kadets.

Observe that $(\mathrm{Id}_X-S_\nu)_{\nu\in A}$ is a~uniformly bounded sequence of operators which converges to $0$ on the dense subspace of $X$ spanned by the set $\{e_n\}_{n=1}^\infty$. Let $\{n_1<n_2<\ldots\}$ be an enumeration of $A$. Then equality \eqref{xnk} holds for every $x\in X$. Since $e_n^\ast$'s are all continuous, the coefficients of every such expansion are uniquely determined, hence the basis in question is in fact an $M$-basis with brackets.
\end{proof}

\section{$\F$-bases with individual brackets}
In view of Theorem \ref{T1}, the inequality $\chi(\F)<\mathfrak{p}$ implies that for all $x\in X$ one may find a~common set $A\in\F$ such that $S_\nu(x)$ converge to $x$ as $\nu\in A$ and $\nu\to\infty$, whereas Example \ref{exx} shows that this is not possible in general. These two facts motivate the following definition.
\begin{definition}
A~sequence $(e_n)_{n=1}^\infty$ of elements of a~Banach space $X$ is called an~$\F$-{\it basis with individual brackets} if it is an $\F$-basis of $X$ and for each $x\in X$ there is a~set $A\in\F$ (possibly depending on $x$) such that $$\lim_{\substack{\nu\to\infty\\ \nu\in A}}\n{S_\nu(x)-x}=0.$$
\end{definition} 

This is a~notion similar to that of $M$-bases with individual brackets which was considered by Kadets \cite{kadets}. Recall that $(e_n)_{n=1}^\infty\subset X$ is called an~$M$-{\it basis with individual brackets} if there is a~sequence of functionals $(e_n^\ast)_{n=1}^\infty$ such that $(e_n,e_n^\ast)_{n=1}^\infty$ is a~Markushevich basis (i.e. it is a~biorthogonal system with $\overline{\mathrm{span}}\{e_n\}_{n=1}^\infty=X$ and $\overline{\mathrm{span}}^{w\ast}\{e_n^\ast\}_{n=1}^\infty=X^\ast$) and for each $x\in X$ there exists a~sequence $n_1<n_2<\ldots$ of natural numbers for which \eqref{xnk} holds true.

Kadets \cite{kadets} showed that the space $\ell_2$ admits an~$M$-basis with individual brackets which is not an~$M$-basis with brackets. The basis exhibited by Kadets was in fact an $\F_s$-basis with a~{\it summable filter} $\F_s$ given by $$\F_s=\Bigl\{A\subset\N\colon\sum_{n\in\N\setminus A}\Bigl((n+1)\sum_{k=1}^n\frac{1}{k}\Bigr)^{-1}<\infty\Bigr\}.$$Since $\F_s$ is a~tall filter, we have $\chi(\F_s)\geq\mathfrak{p}$, which, in light of Theorem \ref{T1}, is not accidental.

Obviously, if $\F$ is a $P$-{\it filter} (i.e. for every countable family $\G\subset\F$ there is a~set $A\in\F$ such that $\abs{A\setminus B}<\omega$ for each $B\in\G$) then every $\F$-basis is an~$\F$-basis with individual brackets. The $\F$-bases exhibited in \cite{connor_ganichev_kadets} and \cite{kadets} are also examples of $\F$-bases with individual brackets. Those construction may be generalised in the following way. 

Let $X$ be a~Banach space with a~Schauder basis $(f_n)_{n=1}^\infty$ and let $(\gamma_n)_{n=1}^\infty$ be a~sequence of non-zero scalars such that the series $\sum_{n=1}^\infty\gamma_nf_n$ diverges. We put $$e_n=\sum_{j=1}^n\gamma_jf_j\quad\mbox{and}\quad e_n^\ast=\frac{1}{\gamma_n}f_n^\ast-\frac{1}{\gamma_{n+1}}f_{n+1}^\ast\quad\mbox{for }n\in\N .$$Then it may be easily checked that $(e_n,e_n^\ast)_{n=1}^\infty$ is a~Markushevich basis of $X$ (the fact that $(e_n^\ast)_{n=1}^\infty$ is a~total subset of $X^\ast$ follows from our supposition on the series $\sum_{n=1}^\infty\gamma_nf_n$). Let $(S_n)_{n=1}^\infty$ and $(T_n)_{n=1}^\infty$ be the partial sum projections corresponding to $(e_n)_{n=1}^\infty$ and $(f_n)_{n=1}^\infty$, respectively. Then, we have 
\begin{equation*}
\begin{split}
S_n(x) &=\sum_{j=1}^ne_j^\ast(x)e_j=\sum_{j=1}^n\Bigl(\frac{1}{\gamma_j}f_j^\ast(x)-\frac{1}{\gamma_{j+1}}f_{j+1}^\ast(x)\Bigr)\sum_{k=1}^j\gamma_kf_k\\
&=\sum_{k=1}^n\sum_{j=k}^n\gamma_k\Bigl(\frac{1}{\gamma_j}f_j^\ast(x)-\frac{1}{\gamma_{j+1}}f_{j+1}^\ast(x)\Bigr)f_k\\
&=\sum_{k=1}^n\Bigl(f_k^\ast(x)-\frac{\gamma_k}{\gamma_{n+1}}f_{n+1}^\ast(x)\Bigr)f_k=T_n(x)-\frac{f_{n+1}^\ast(x)}{\gamma_{n+1}}\sum_{j=1}^n\gamma_jf_j,
\end{split}
\end{equation*}
whence
$$\n{S_n(x)-T_n(x)}=\frac{\abs{f_{n+1}^\ast(x)}}{\abs{\gamma_{n+1}}}\left\|\sum_{j=1}^n\gamma_jf_j\right\|\quad\mbox{for }x\in X\mbox{ and }n\in\N.$$Consequently, $(e_n)_{n=1}^\infty$ is an $\F$-basis of $X$, where $\F$ is the~filter generated by the sets of the form $$
\left\{n\in\N\colon \frac{\abs{f_{n+1}^\ast(x)}}{\abs{\gamma_{n+1}}}\left\|\sum_{j=1}^n\gamma_jf_j\right\|<1\right\}\quad (x\in X),
$$
provided only that the intersection of any finite number of these sets is infinite (this condition, jointly with the fact that $(e_n,e_n^\ast)_{n=1}^\infty$ is biorthogonal, guarantees that every expansion with respect to $(e_n)_{n=1}^\infty$ is unique). 

The reason why all the $\F$-bases arising in this manner are $\F$-bases with individual brackets is that for any $x\in X$ and $n\in\N$ the difference $S_n(x)-T_n(x)$ involves only the $(n+1)$st coordinate of $x$ with respect to the basis $(f_n)_{n=1}^\infty$. We shall see that this is a~special case of a~more general result.

To formulate the announced result we need a~piece of notation. Namely, if $(f_n,f_n^\ast)_{n=1}^\infty$ is a~Schauder basis of a~Banach space $X$ and $T\colon X\to X$ is a~finite-rank operator which may be written as $$T(x)=\sum_{j=1}^kf_{n_j}^\ast(x)x_j\quad (x\in X),$$ with some non-zero $x_1,\ldots ,x_k\in X$ and some natural numbers $n_1<\ldots <n_k$, then we write $\supp(T)$ for the set $\{n_1,\ldots ,n_k\}$. If $A,B\subset\N$ are finite then we write $A<B$ provided $\max A<\min B$.

\begin{theorem}\label{PP}
Let $(e_n)_{n=1}^\infty$ be an~$\F$-basis of a~Banach space $X$ with partial sum projections $(S_n)_{n=1}^\infty$. Suppose that there is a~Schauder basis $(f_n)_{n=1}^\infty$ of $X$ with partial sum projections $(T_n)_{n=1}^\infty$ such that for some set $\{n_1<n_2<\ldots\}\in\F$ we have $$\supp(S_{n_1}-T_{n_1})<\supp(S_{n_2}-T_{n_2})<\ldots .$$Then $(e_n)_{n=1}^\infty$ is an~$\F$-basis with individual brackets.
\end{theorem}
\begin{proof}
By the definition of $\F$-basis, the set $$D_x:=\bigl\{n\in\N\colon\n{S_n(x)-T_n(x)}<1\bigr\}$$belongs to $\F$ for each $x\in X$.

Fix any $x\in X$. We shall find $y\in X$ such that for arbitrarily large $M>0$ the inequality
\begin{equation}\label{yy}
\n{S_{n_j}(y)-T_{n_j}(y)}\geq M\cdot\n{S_{n_j}(x)-T_{n_j}(x)}
\end{equation}
holds true for all but finitely many $j\in\N$. Then, by putting $A=\{n_1,n_2,\ldots\}\cap D_y$, we would obtain a~set $A\in\F$ for which $$\lim_{\substack{\nu\to\infty\\ \nu\in A}}\n{S_\nu(x)-T_\nu(x)}=0,$$which would in turn imply that $x\in\w{X}_A$.

Define a sequence $1\leq r_1<r_2<\ldots$ by $$r_j=\max\supp(S_{n_j}-T_{n_j}).$$Then for each $j\in\N$ we have 
\begin{equation}\label{ST}
S_{n_j}(z)-T_{n_j}(z)=\sum_{i=r_{j-1}+1}^{r_j}f_i^\ast(z)x_{i,j}\quad (z\in X)
\end{equation}
for some $x_{i,j}\in X$ (we put $r_0=0$). We claim that there exists a~sequence $1\leq\nu_1<\nu_2<\ldots$ of natural numbers such that if we define a~sequence $(\lambda_n)_{n=1}^\infty$ by saying that $\lambda_n=k$ if and only if $r_{\nu_{k-1}}<n\leq r_{\nu_k}$ (where $\nu_0=0$), then the series
\begin{equation}\label{y_def}
y:=\sum_{n=1}^\infty\lambda_nf_n^\ast(x)f_n
\end{equation}
converges in $(X,\n{\cdot})$. Indeed, we may define $(\nu_j)_{j=1}^\infty$ inductively by first choosing $\nu_1\geq 1$ such that for any $\nu_1\leq p\leq q$ we have $$\left\|\sum_{n=p}^qf_n^\ast(x)f_n\right\|<2^{-3}$$ and, after defining $1\leq\nu_1<\ldots<\nu_{j-1}$, we pick $\nu_j>\nu_{j-1}$ such that for any $\nu_j\leq p\leq q$ we have $$\left\|\sum_{n=p}^qf_n^\ast(x)f_n\right\|<(j+1)^{-3}.$$Now, if $(\lambda_n)_{n=1}^\infty$ is defined as above, then for any $\e>0$ we may find $k\in\N$ so large that $\sum_{j\geq k}j^{-2}<\e$. Then for any $m>\nu_{k-1}$ we have
\begin{equation*}
\begin{split}
\left\|\sum_ {n=\nu_{k-1}+1}^m\lambda_nf_n^\ast(x)f_n\right\| &\leq\sum_{j=k}^\infty\left\|\sum_{n=\nu_{j-1}+1}^{\nu_j}\lambda_nf_n^\ast(x)f_n\right\|\\
&=\sum_{j=k}^\infty j\cdot\left\|\sum_{n=\nu_{j-1}+1}^{\nu_j}f_n^\ast(x)f_n\right\|<\sum_{j=k}^\infty j^{-2}<\e,
\end{split}
\end{equation*}
which shows that the series given by \eqref{y_def} converges.

Now, fix any $j\in\N$. There is a~unique $k\in\N$ such that $(r_{j-1},r_j]\subset (r_{\nu_{k-1}},r_{\nu_k}]$ and for any $r_{j-1}<i\leq r_j$ we have $r_{\nu_{k-1}}<i\leq r_{\nu_k}$. Hence for any such $i$ we have $\lambda_i=k$. Then, by \eqref{y_def}, we get $f_i^\ast(y)=\lambda_if_i^\ast(x)=kf_i^\ast(x)$ for $r_{j-1}<i\leq r_j$, so formula \eqref{ST} yields $$\n{S_{n_j}(y)-T_{n_j}(y)}=k\cdot\n{S_{n_j}(x)-T_{n_j}(x)}.$$Therefore, inequality \eqref{yy} is valid whenever $j$ satisfies $$(r_{j-1},r_j]\subset\bigcup_{k\geq M}(r_{\nu_{k-1}},r_{\nu_k}],$$which is true for all but finitely many $j\in\N$.
\end{proof}

In view of Theorem \ref{PP}, it is tempting to ask whether there exists any $\F$-basis that is not an~$\F$-basis with individual brackets.

\bibliographystyle{amsplain}

\end{document}